\documentclass[a4paper,11pt, leqno]{amsart}
\usepackage{graphicx}
\usepackage{amssymb}
\numberwithin{equation}{section}
\newtheorem{Theorem}{Theorem}[section]

\newtheorem{Fact}[Theorem]{Fact}

\newtheorem{Proposition}[Theorem]{Proposition}

\theoremstyle{remark}
\newtheorem{Remark}[Theorem]{Remark}

\theoremstyle{definition}
\newtheorem{Definition}[Theorem]{Definition}

\newtheorem*{acknowledgements}{Acknowledgements}
\newcommand{\R}{{\mathbb R}}

\newcommand{\mc}[1]{{\mathcal #1}}
\newcommand{\mb}[1]{{\mathbf #1}}

\renewcommand{\phi}{\varphi}
\renewcommand{\epsilon}{\varepsilon}
\newcommand{\op}[1]{{\operatorname{ #1}}}

\renewcommand{\phi}{\varphi}
\renewcommand{\epsilon}{\varepsilon}

\renewcommand{\phi}{\varphi}

\newcommand{\mathsym}[1]{{}}
\newcommand{\unicode}[1]{{}}

\title{
Mannheim-d'Ocagne-Koenderink type
formulas 
for asymptotic directions
}

\author{Toshizumi Fukui}
\address[Toshizumi Fukui]
{Department of Mathematics,
Saitama University,
255, Shimo-Okubo,
Saitama, 338-8570,
Japan}
\email{tfukui@rimath.saitama-u.ac.jp}

\author{Atsufumi Honda}
\address[Atsufumi Honda]
{Department of Applied Mathematics, 
Yokohama National University,
Yokohama 240-8501, Japan}
\email{honda-atsufumi-kp@ynu.ac.jp}

\author{Masaaki Umehara}
\address[Masaaki Umehara]{
  Department of Mathematical and Computing Sciences,
  Institute of Science Tokyo,
  2-12-1-W8-34, O-okayama Meguro-ku,
  Tokyo 152-8552, Japan}
\email{umehara@comp.isct.ac.jp}

\usepackage[active]{srcltx}

\begin{document}

\maketitle

\begin{abstract}
We consider a surface embedded in the Euclidean $3$-space
and fix a tangential vector $\mb v$ at a given point $p$ on the surface.
In this paper, we first review a history of the formula obtained by 
Mannheim, d'Ocagne and Koenderink, which 
asserts that the Gaussian curvature of the 
surface at $p$
can be obtained if one 
knows
\lq\lq the normal curvature  at $p$
with respect to ${\mb v}$"
and \lq\lq the curvature of the contour line $\Gamma$ of the surface at $p$"
with respect to the orthogonal projection induced by $\mb v$.
Unfortunately, this formula does not work when $\mb v$ points in
an asymptotic direction.
When $\mb v$ is just the case,
we  give
analogues of the formula,
which include an invariant of 
cusp singular points of $\Gamma$.
\end{abstract}

\section{Introduction}
Let $f:U\to \R^3$ be an embedding, where $U$ is a simply connected
domain of $\R^2$.
Then we can fix a unit normal vector field $\nu$ of $f$
defined on $U$.
We fix $p\in U$ and a non-zero tangent vector
$\mb v\in T_pU$ at $p\in U$.
We set $V:=df_p(\mb v)$. 
We consider the following 
three mutually orthogonal planes in $\R^3$
passing through the point $f(p)$;
\begin{itemize}
\item the {\it tangential plane} $\mb T_p$ 
which is perpendicular to the unit normal vector $\nu_p$ of $f$ at $p$,
\item the plane $N_V$ spanned by $\nu_p$ and $V$, called the {\it normal plane
containing $V$},
 \item the {\it $V$-plane $\Pi_V$}, which is perpendicular to the vector $V$.
\end{itemize}
For example, if $f(x,y):=(x,y,x y)$, $p:=(0,0)$ and $\mb v:=(\partial/\partial y)_p$,
then $\mb T_p$ is the $xy$-plane, $N_V$ is the $yz$-plane and
$\Pi_V$ is the $xz$-plane.

Let $\pi_{V}:\R^3\to \Pi_{V}$ be the orthogonal
projection to the $V$-plane.
We denote by $K_p$ $($resp. $\kappa_p(\mb v))$ 
the Gaussian curvature $($resp. the normal curvature
with respect to the direction $\mb v$) of $f$ at $p$.
If $\kappa_p(\mb v)\ne 0$, then 
the image of the projection $g:=\pi_{V}\circ f$
is bounded by a regular curve $\Gamma$ in 
the plane $\Pi_V$,
which is called the {\it contour line} of $f$ 
with respect to $V$.
For example, if we consider an elliptic paraboloid
or a hyperbolic paraboloid given by
$$
f_{\pm}(x,y):=(x,y,2x^2\pm y^2),
$$
then by setting $p:=(0,0)$ and $V:=(0,1,0)$,
the figures of $f_\pm$ and
$g_\pm=\pi_{V}\circ f_\pm$ are
indicated in Fig.~\ref{fig:1}.

\begin{figure}[h!]
\begin{center}
\includegraphics[height=3.5cm]{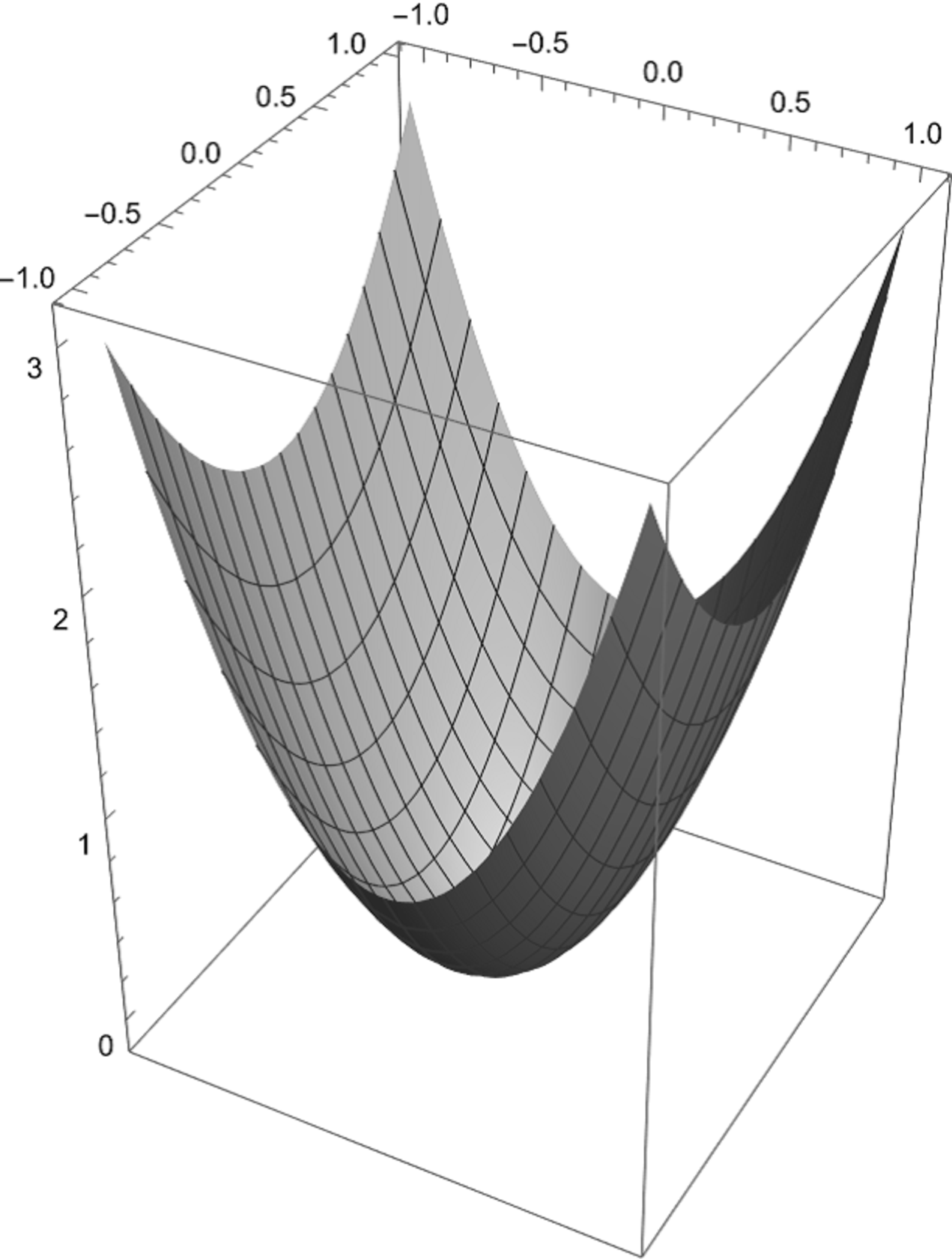}\qquad
\includegraphics[height=3.4cm]{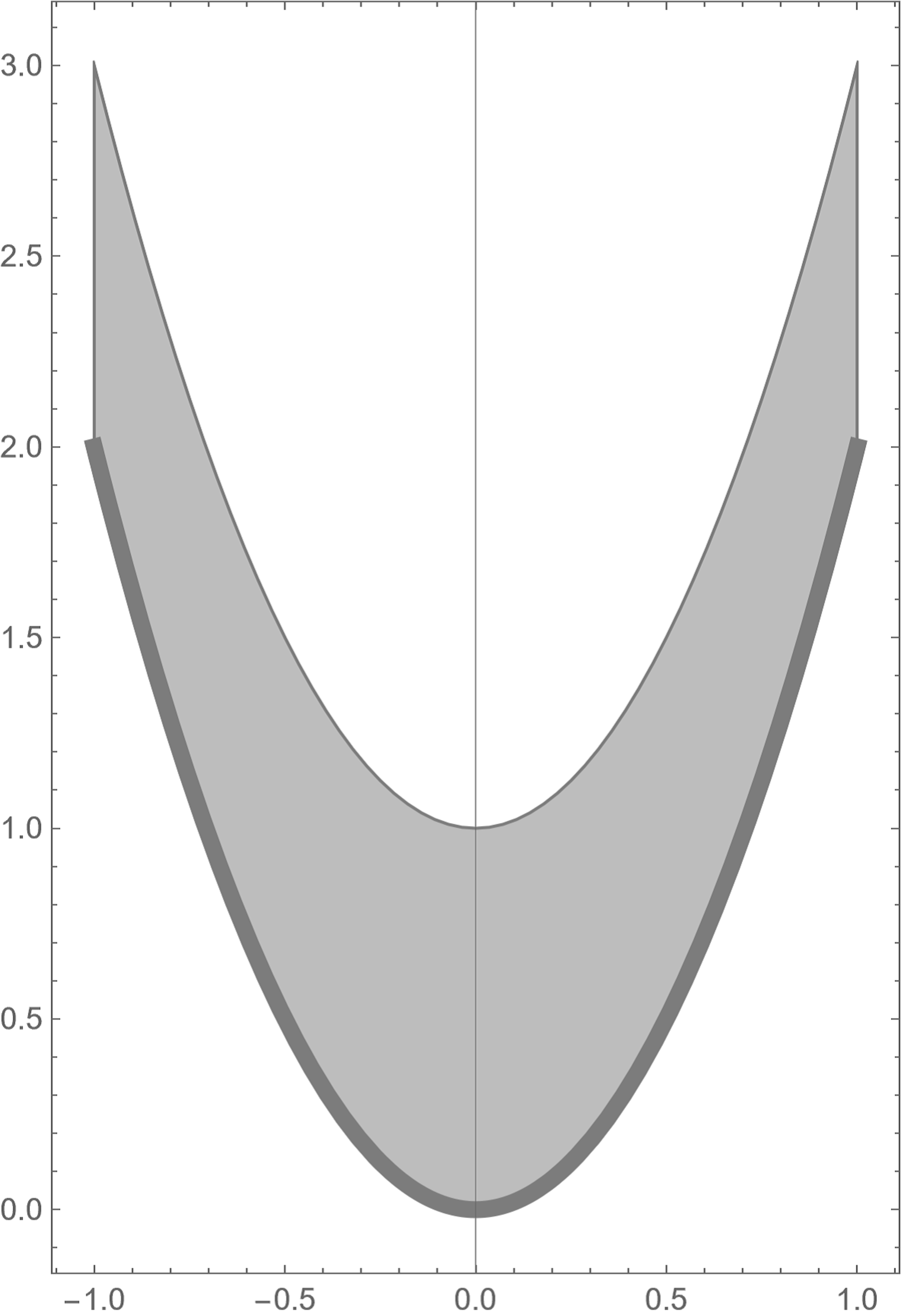}\quad \quad
\includegraphics[height=3.5cm]{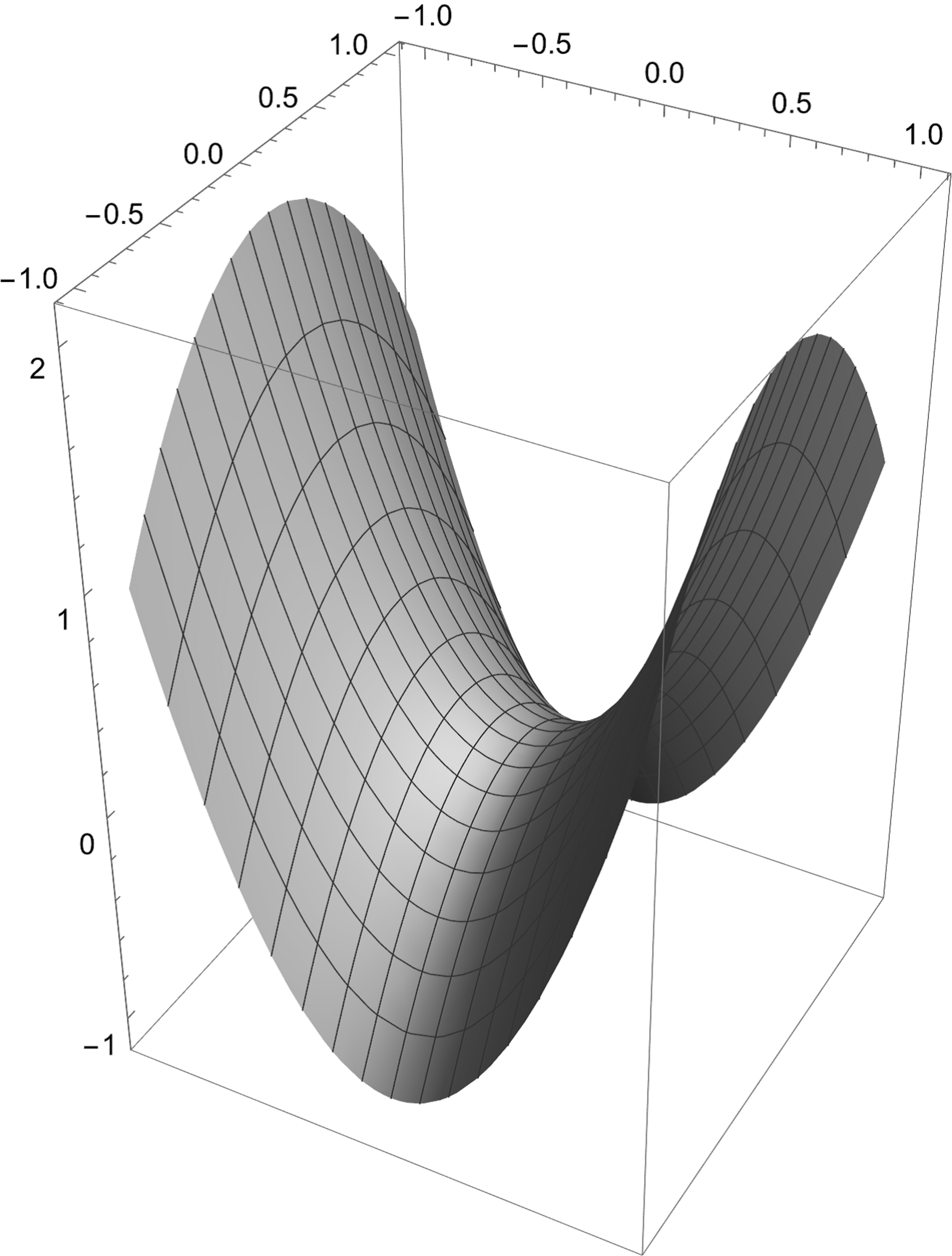}\quad
\includegraphics[height=3.4cm]{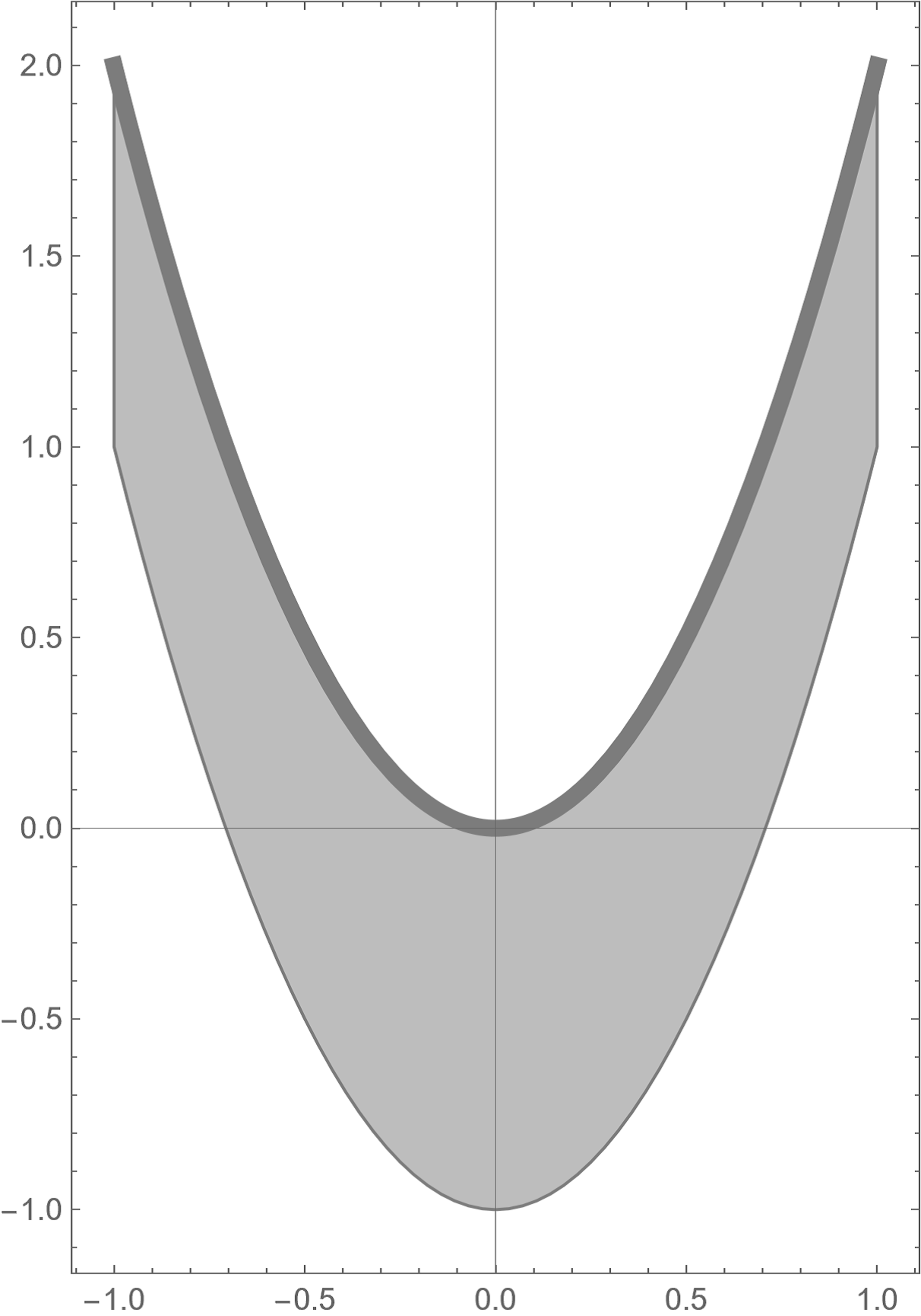}
\end{center}
\caption{The image of $f_+$ and its projection $g_+$ (left)
 and the image of $f_-$ and its projection $g_-$ (right)
}  \label{fig:1}
\end{figure}

It is well-known that 
the normal curvature $\kappa_p(\mb v)$  of
the surface $f$ at $p$ with respect to the tangential direction
$\mb v$ is written as
\begin{equation}\label{eq:226}
\kappa_p(\mb v)=
\lambda_1(p)\cos^2\phi+
\lambda_2(p)\sin^2\phi,
\end{equation}
where $\lambda_1$ and $\lambda_2$ are
principal curvature functions and
$\phi$ is the angle between the tangential direction $\mb v$
and the principal direction with respect to $\lambda_1$,
which was established by Euler in 1760.

We now assume that the Gaussian curvature $K_p$ of $f$ at $p$
does not vanish. 
In 1866, Mannheim \cite[Pages 307--310]{M0} gave the following formula of
the curvature radius $R_p$ of 
the contour line  
$\Gamma$ at $g(p)=\pi_{V}\circ f(p)$;
$$
R_p=
\left|
\frac1{\lambda_1(p)}\sin^2\phi+
\frac1{\lambda_2(p)}\cos^2\phi
\right|,
$$
which is also found in his book \cite{M}.
This formula can be considered as the dual of 
\eqref{eq:226} 
(cf. Blaschke \cite[Page 118]{B},
see also Kruppa \cite{Kr} 
where the duality is 
explained from a more geometric point of view). 
In fact, 
the geodesic curvature $\mu(p)=\pm 1/R_p$ of $\Gamma$ at $g(p)$
of the contour line  $\Gamma$ at $g(p)=\pi_{V}\circ f(p)$
satisfies
\begin{align*}
\mu(p)\kappa_p(\mb v)
=
\pm \lambda_1(p)\lambda_2(p)=\pm K_p.
\end{align*}
In particular,
\begin{equation}\label{eq:271}
|K_p|=|\mu(p)\kappa_p(\mb v)|
\end{equation}
holds.  Although Mannheim  might recognize this,
 d'Ocagne \cite{Oc} is the first person who 
pointed out \eqref{eq:271}.
Moreover, he gave a new proof of it.
 After that, in 1985, 
Koenderink \cite{K} rediscovered it
and pointed out the convexity and concavity of
the projection. 
In fact, he actually wrote in \cite{K} that
{\it a convexity of the contour corresponds 
to a convex patch of the surface, and a concavity to a saddle-shaped patch}.
He also explained this subject in his book \cite{K2} 
from the view point of geometric shape generation.
In other word, he found the formula
\begin{equation}\label{eqKE}
K_p=\mu(p)\kappa_p(\mb v),
\end{equation}
where $\mu(p)$ is positive (resp. negative)
when the image of $g:=\Pi_V\circ f$ is
locally convex (resp. concave) at $p$.
So, we call the formula
the {\it Mannheim-d'Ocagne-Koenderink formula}.
A modern proof of \eqref{eqKE}
is given in \cite[Chapter 4]{SUY}.

\begin{Remark}
In almost all recent papers, this formula was cited as
Koenderink's formula.
In fact, the authors also used this naming, but very recently,
Professor Farid Tari informed us of the work of
d'Ocagne. 
In Anjyo-Kabata \cite{AK}, this formula was called
d'Ocagne formula.
\end{Remark}

The formula does not apply when 
$\mb v$ is an asymptotic direction.
In just such a case, 
$\Gamma$ has a singularity. Using an 
information at the singular point,
we will give similar formulas (cf. Theorem \ref{thm:292}).
For example, if 
$
f_0(x,y):=(x,y,xy)
$
 and $V:=(0,1,0)$,
then the singular set of
$g_0:=\pi_{V}\circ f_0$ is the $y$-axis
and the image of the singular points of $g_0$ is just 
the origin $(0,0)$ (see Fig.~\ref{fig:2}, top row).
However, since the singular points of $f_0$ are degenerate 
into a single point, we consider the 
following map including the perturbation term $y^3$.
$$
f_1(x,y):=(x,y,xy+y^3).
$$
If we set $V:=(0,1,0)$,
then the singular set of
$g_1:=\pi_{V}\circ f_1$ is 
the parabola $x=-3y^2$ in the domain of 
definition of $f_1$
and so the contour line of $f_1$ is given by
$$
\Gamma(y):=g_1(-3y^2,y)=(-3y^2,-2y^3),
$$
 which is
a plane curve in the $xz$-plane having a cusp singular point at $o:=(0,0)$
(cf. Fig.~\ref{fig:2}, bottom row).
The origin $o$ gives a Whitney cusp $($see the appendix$)$
of the map $g_1$.

\begin{figure}[h!]
\begin{center}
\includegraphics[height=4.0cm]{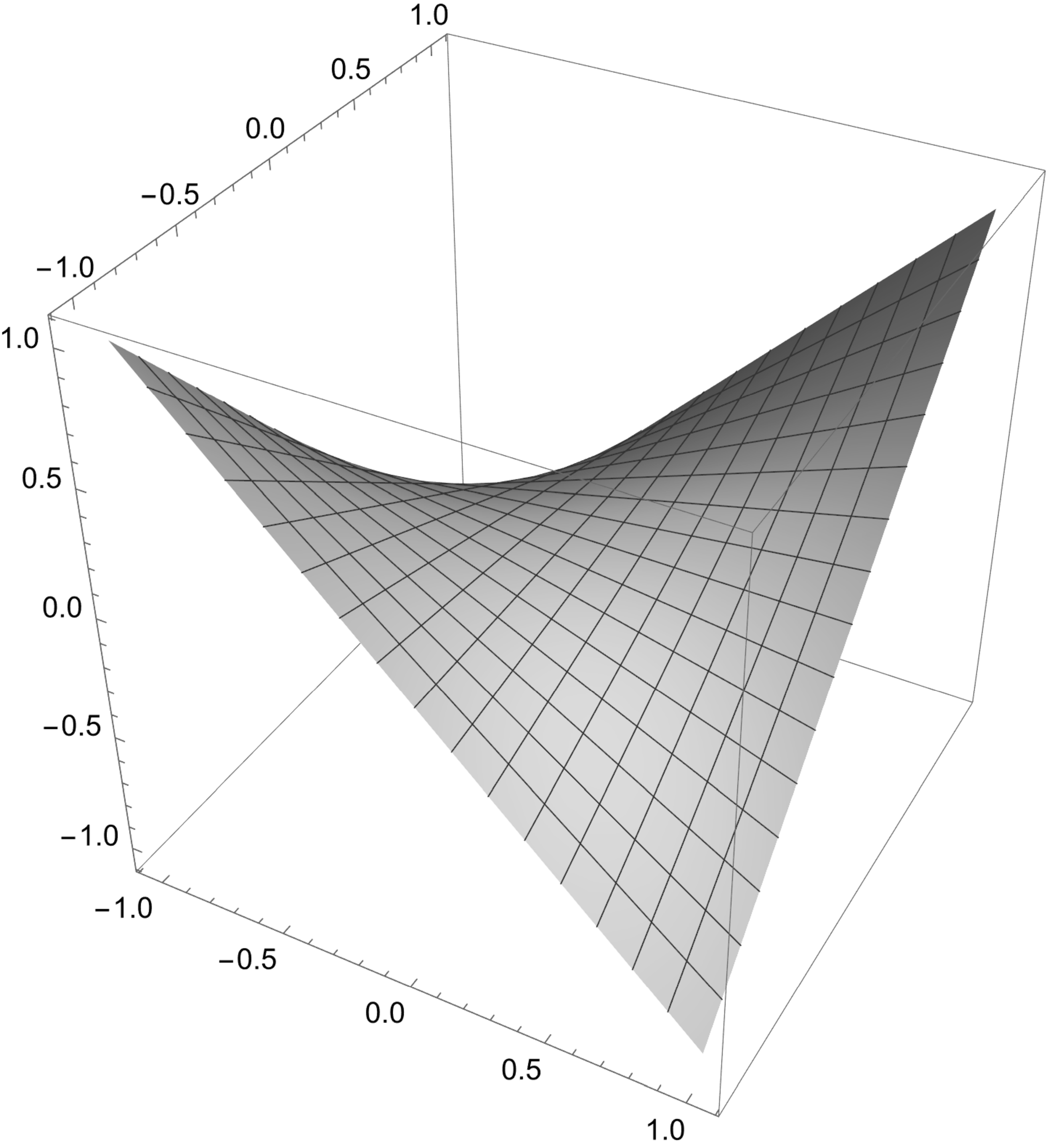}\qquad
\includegraphics[height=3.7cm]{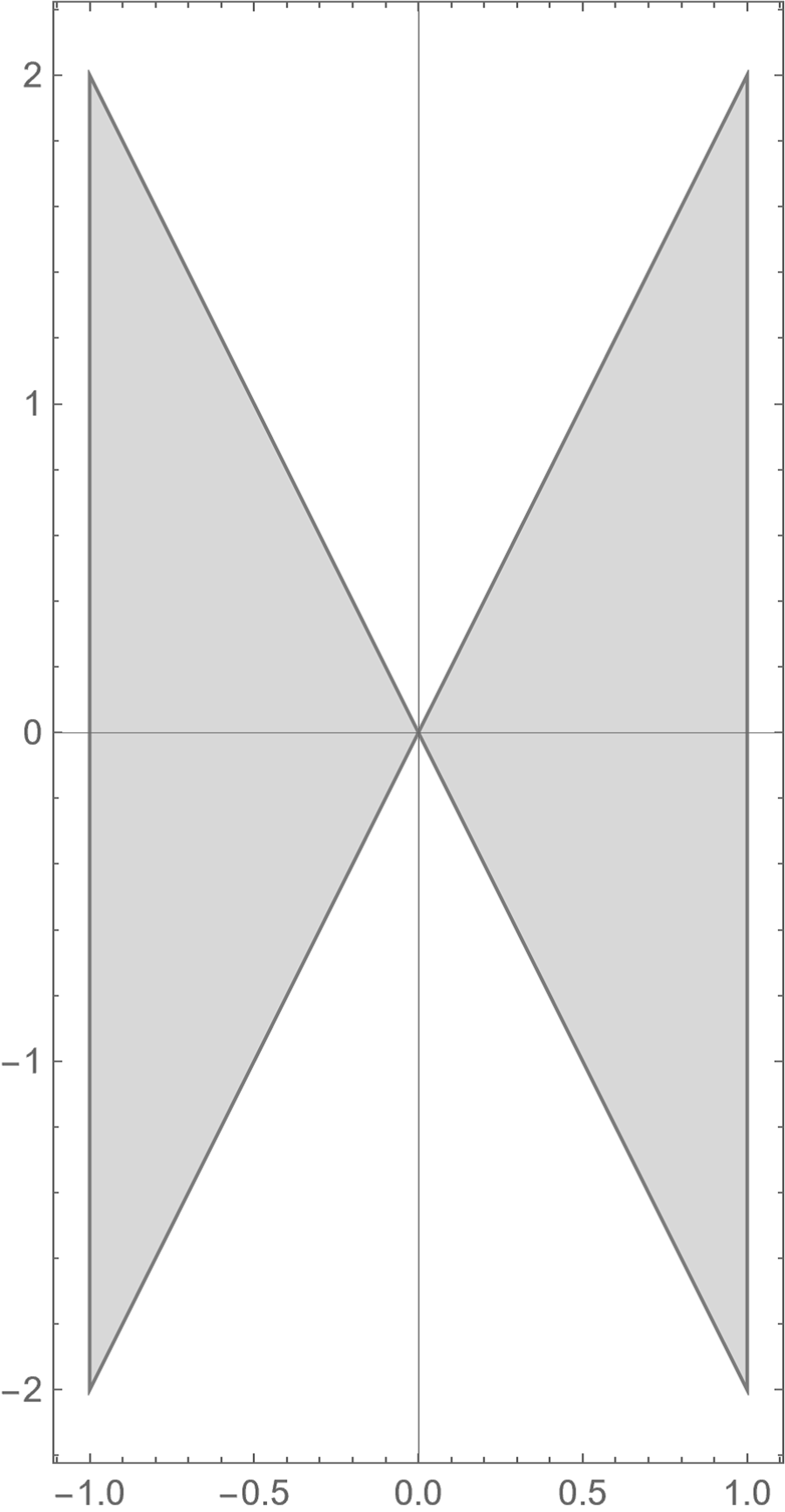}\qquad \qquad
\includegraphics[height=4.0cm]{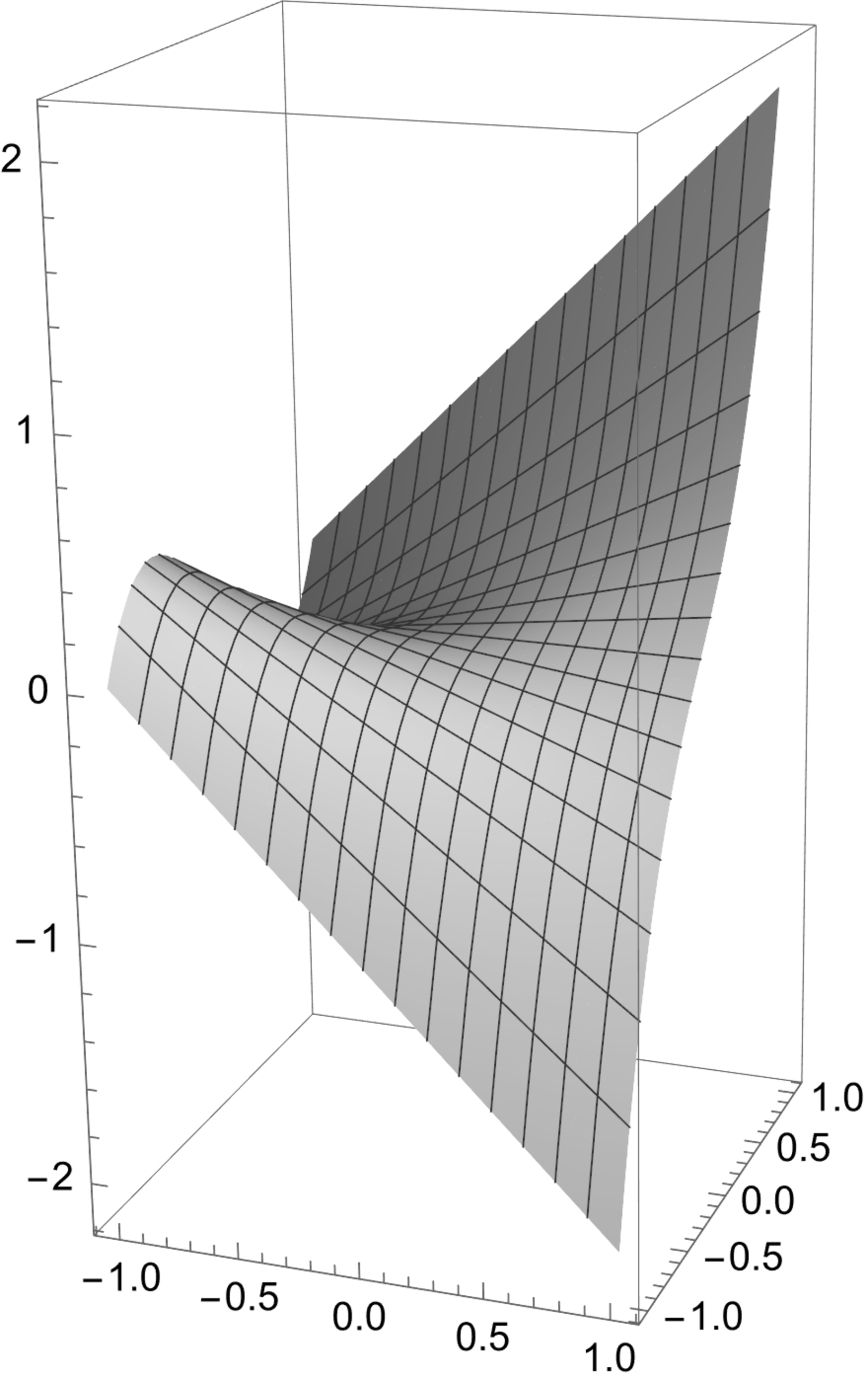}\qquad
\includegraphics[height=3.7cm]{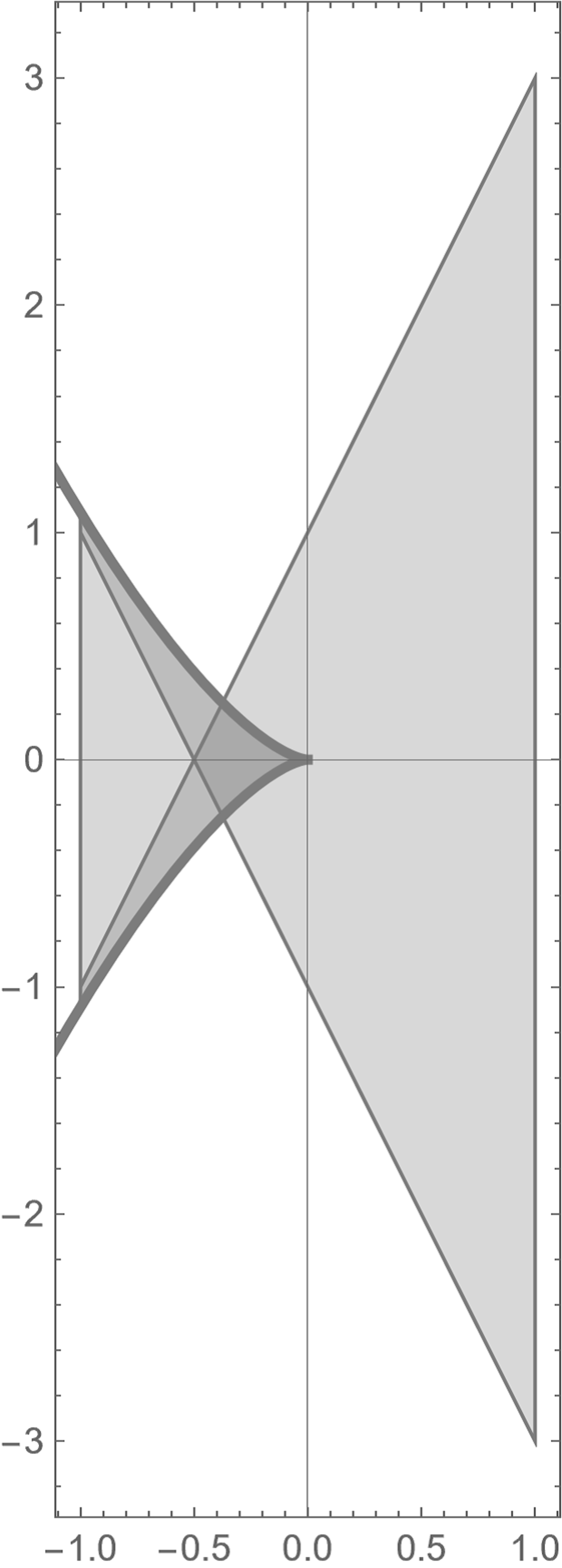}
\end{center}
\caption{The image of $f_0$ and its projection $g_0$ (left)
 and the image of $f_1$ and its projection $g_1$ (right)
} 
\label{fig:2}
\end{figure}

We consider an embedding $f:U\to \R^3$ defined on a domain $U$ in the
$uv$-plane $\R^2$. We fix $p\in U$ at which 
the Gaussian curvature $K_p$ is negative.
Since $f$ is an embedding, we often identify a point of $U$ with its image 
by $f$ and a tangent vector to $U$ with a tangent vector to the image, 
for the sake of simplicity.
Let $\mb v\in T_pU$ be
a tangent vector of $f$ at $p$ giving an asymptotic direction.

\begin{Definition}
We call the section $\mc S_p$ of the image of $f$
by the normal plane $N_V$ {\it the asymptotic normal section} at $p$. 
Then, $\mc S_p$ is the curve with an
inflection point at $f(p)$.
The derivative of 
the curvature function of the curve $\mc S_p$
with respect to the arc-length parameter is denoted by 
$\rho_p(\mb v)$ and is called {\it vertical torsion} at $p$
with respect to the asymptotic direction $\mb v$.
\end{Definition}

By definition, the vertical torsion $\rho_p(\mb v)$ gives
the information of the surface at $p$ with respect to the
normal plane $N_V$.
We may assume that $f(p)$ is the origin in $\R^3$
and the vector 
$$
\mb e_3:=(0,0,1)\qquad 
(\text{resp.}\quad \mb e_2:=(0,1,0))
$$
points in the normal direction of $f$ at $p$
(resp. the same direction as $df_p(\mb v)$).
Then $N_V$ is the $yz$-plane and the graph of one of the cubic 
polynomials $z=\pm \rho_p(\mb v) y^3/6$ gives the
best approximation of $\mc S_p$ at $y=0$
in the $yz$-plane among cubic polynomials in $y$.

\begin{Definition}
Let $I$ be an open interval containing the origin of $\R$
and $\sigma(t)$ ($t\in I$)
a regular curve in the domain $U(\subset \R^2)$ satisfying $\sigma(0)=p$.
Then $f\circ \sigma$ (resp. $\sigma$)
is called an {\it asymptotic tangential curve}
(resp. a {\it generator of asymptotic tangential curve})
at $p$ (with respect to $\mb v$)
if $f\circ \sigma(I)$ lies in the tangent 
plane $\mb T_p$ of $f$ passing through $f(p)$
satisfying $\sigma'(0)=\mb v$, where $\sigma'=d\sigma/dt$.
We give the orientation of $\mb T_p$
so that $(\mb v, \nu_p\times \mb v)$ gives
a positive frame.
Since $\sigma$ is a regular curve in the oriented plane $\mb T_p$,
we can define $\alpha_p(\mb v)$ by
the curvature function of 
the curve $\sigma$ at $p$ as a plane curve,
which is called the
{\it horizontal curvature} at $p$ with respect
to the asymptotic direction $\mb v$.
\end{Definition}

By definition, 
the horizontal curvature $\alpha_p(\mb v)$ gives
the information of the surface at $p$ with respect to the
tangential plane $\mb T_p$.
On the other hand, there exists a regular curve $\tau:I\to U$ 
defined on an open interval $I$ containing the origin in $\R$
satisfying the following:
\begin{enumerate}
\item $\tau(0)=p$ and $\tau'(0)$ is a positive scalar multiplication of $\mb v$,
and
\item for each $t\in I$, the velocity vector $\tau'(t)$ points 
in the asymptotic direction at $\tau(t)\in U$ of $f$. 
\end{enumerate}
The curve $\tau$ is called the {\it asymptotic curve} of $p$ with respect to
the asymptotic direction $\mb v$.

\begin{Definition}
We set $\hat \tau(t):=f\circ \tau(t)$. 
The curvature $\beta_p(\mb v)(>0)$ 
(resp. the torsion $\delta_p(\mb v)(\in \R)$)
of the space curve $\hat \tau$ at $t=0$ is called
the {\it asymptotic curvature} (resp. {\it the asymptotic torsion}) of
$f$ at $p$ with respect to the asymptotic direction $\mb v$.
\end{Definition}

It holds the following remarkable property
(cf. \cite[Chapter~4]{Sp}):

\begin{Fact}[Beltrami-Enneper]\label{prop:B}
$|\delta_p(\mb v)|$  
coincides with $\sqrt{-K_p}$.
\end{Fact}

We will give a proof of the fact in Section 3. 
The two invariants $\beta_p(\mb v)$ and $\delta_p(\mb v)$
are related to the invariants $\alpha_p$ and $\rho_p$ as follows:

\begin{Proposition}\label{prop:A}
Let $f:U\to \R^3$ be an embedding and
$p$ a point on $U$ at which the Gaussian 
curvature $K_p$ of $f$ is negative.
Let $\mb v\in T_pU$ be a tangent vector of $f$  at $p$ 
giving an asymptotic direction. Then
$$
|\alpha_p(\mb v)|=\frac23\beta_p(\mb v),\quad
|\delta_p(\mb v)|\beta_p(\mb v)=\frac{|\rho_p(\mb v)|}2
$$
holds. In particular $($cf. Fact \ref{prop:B}$)$, the following formula holds
\begin{equation}\label{eq:307}
K_p=-\frac{\rho_p(\mb v)^2}{9\alpha_p(\mb v)^2}.
\end{equation}
\end{Proposition}

This proposition will be proved 
in Sections 3 and 4.
Since $K_p$ 
is written as a product of the invariants 
with respect to the planes $N_V$ and $\mb T_p$,
\eqref{eq:307} can be seen as a
variant of the Mannheim-d'Ocagne-Koenderink
formula for asymptotic directions 
as well as the other two formulas 
\eqref{eq:307a} 
and
\eqref{eq:307b}
which we will discuss later.

We  next consider a smooth map
$
\Gamma:(-\epsilon,\epsilon)\to \R^2 \,\, (\epsilon>0)
$ 
giving a cusp at $t=0$. Then the value
\begin{equation}\label{eq:305}
\mu_\Gamma:=\frac{\det(\Gamma''(0),\Gamma'''(0))}{|\Gamma''(0)|^{5/2}}
\end{equation}
is called the {\it cuspidal curvature} of $\Gamma$ at $t=0$,
which is an invariant of $\Gamma$ at the cusp $t=0$, where \lq\lq $\det$"
denotes the determinant when each component is considered a column vector. 
It is well-known that  the absolute value of the curvature $\kappa(t)$ ($t\ne 0$)
at a point $\Gamma(t)$ is the inverse of the radius of the
best approximated circle at the point.
The above invariant $\mu_\Gamma$
is introduced by 
imitating the property of $\kappa(t)$,
that is, the absolute value $\mu_\Gamma$ at the cusp of $\Gamma$
is equal to $\sqrt{a}$
when the best approximated cycloid at $\Gamma(0)$
the cycloid obtained by rolling a circle of radius $a$ ($a>0$).
(cf. \cite[Example 1.3.13]{SUY}).
The relationship between $\mu_\Gamma$ and the behavior of $\kappa(t)$
is explained in \cite[Theorem 1.1]{SU}.
The following is the main result:

\begin{Theorem}\label{thm:292}
Let $f:U\to \R^3$ be an embedding and
$p$ a point on $U$ at which the Gaussian 
curvature $K_p$ of $f$ is negative. 
If $\mb v\in T_pU$ is a non-zero vector 
giving an asymptotic direction of $f$ at $p$,
then the following four assertions are
equivalent: 
\begin{enumerate}
\item The horizontal curvature $\alpha_p(\mb v)$ does not vanish.
\item The vertical torsion $\rho_p(\mb v)$ does not vanish.
\item The contour line $\Gamma$ of 
$g:=\pi_{V}\circ f$
gives a cusp in the plane $\Pi_{V}$
$($like as in Fig.~\ref{fig:2}, right$)$,
where $V:=df_p(\mb v)$.
\item The point $p$ is a Whitney cusp $($see the appendix$)$
of the map $g$.
\end{enumerate}
Moreover, let $\omega_p(\mb v)(:=\mu_\Gamma)$ be the
cuspidal curvature of $\Gamma$ at $\pi_{V}\circ f(p)$. 
Under the assumption that one of the above these three conditions hold, 
the following two equations hold;
\begin{align}
\label{eq:307a}
K_p^3&=-\frac{\rho_p(\mb v)^2 \omega_p(\mb v)^4}{16}, \\
\label{eq:307b}
K_p&=-\frac{3}4 |\alpha_p(\mb v)|\omega_p(\mb v)^2.  
\end{align}
\end{Theorem}

We have noted that
\eqref{eq:307}
is the formula of $K_p$
with respect to the planes $N_V$ and $\mb T_p$.
On the other hand, 
the first formula \eqref{eq:307a} is related to the planes $N_V$ and $\Pi_V$,
and the second formula \eqref{eq:307b} 
is related to the planes $\mb T_p$ and $\Pi_V$.
Hence,
all of these three formulas can be considered 
as analogues of the Mannheim-d'Ocagne-Koenderink formula 
with respect to asymptotic directions.
Moreover,
\eqref{eq:307a} and \eqref{eq:307b}
represent the Gaussian curvature $K_p$
based on the geometric invariants of
the asymptotic direction of the surface before projection
and $\omega_p(\mb v)$.
We remark that a different 
formula giving a geometric meaning of $K_p/\omega_p(\mb v)$
has been shown in \cite{FHS}. 

\begin{Remark}
By \eqref{eq:307}, \eqref{eq:307a} and \eqref{eq:307b},
the four invariants
$K_p$ $\alpha_p(\mb v)$, $\omega_p(\mb v)$ and
$\rho_p(\mb v)$ are determined if any two of them are given. 
\end{Remark}

\section{General Settings}
We fix an embedding $f:U\to \R^3$
and a point $p\in U$.
Without loss of generality,
we may assume $p:=o$ ($o:=(0,0)$).
Then we can write
$
f(x,y)=(x,y,\phi(x,y))
$ 
where $\phi(x,y)$ is a smooth function
satisfying
\begin{equation}\label{eq:348}
p=(0,0),\qquad
\phi(o)=\phi_x(o)=\phi_y(o)=0,
\end{equation}
where $\phi_x:=\partial \phi/\partial x$
and $\phi_y:=\partial \phi/\partial y$.
We suppose that 
$
\mb v:=\partial/\partial y
$
is the asymptotic direction at $p$.
Then, we have
\begin{equation}\label{eq:383}
\phi_{yy}(o)=0.
\end{equation}

\begin{Proposition}\label{prop:386}
$K_p=-\phi_{xy}(o)^2$ holds.
\end{Proposition}

\begin{proof}
By
\eqref{eq:383},
we have 
\begin{equation}\label{eq:532}
K_p=\phi_{xx}(o)\phi_{yy}(o)-\phi_{xy}(o)^2=-\phi_{xy}(o)^2,
\end{equation}
proving the assertion.
\end{proof}
We remark that
\begin{equation}
N(x,y):=(-\phi_x(x,y),-\phi_y(x,y),1)
\end{equation}
gives a normal vector field of $f$. 
Since the geodesic curvature $\kappa(x)$ of the regular
curve of the $xy$-plane given by
the graph $y=ax^3$ at $x=0$
satisfies $d\kappa /ds(0)=6a$ (where $s$ is the arclength parameter),
the following assertion is obtained:

\begin{Proposition}\label{prop:R}
The vertical torsion $\rho_p(\mb v)$ 
of $f$ at $p$ with respect to the
asymptotic direction $\mb v$ is given by
$\rho_p(\mb v)=\phi_{yyy}(o)$.
\end{Proposition}

\section{Asymptotic curvatures and asymptotic torsions}
Let $\tau$ be an asymptotic curve passing through $o$ 
with respect to $\mb v$.
We can write 
$$
\tau(y)=(x(y),y),\qquad x'(0)=0.
$$
If we set $\hat \tau(y):=f\circ \tau(y)$, then 
\begin{align*}
\hat \tau'&=(x',1,\phi_x x'+\phi_y), \\
\hat \tau''&=(x'',0,\phi_x x'' +\phi_{xx} x'^2+2\phi_{xy}x'+\phi_{yy}).
\end{align*}
Moreover, since $x'(0)=0$ and $\phi_x(o)=0$, we have
\begin{equation}\label{eq:405}
\hat \tau'''(0)=(x'''(0),0, 3x''(0)\phi_{xy}(o)+\phi_{yyy}(o)).
\end{equation}
By \eqref{eq:348} with $x'(0)=0$, we have
$\hat \tau'(0)=(0,1,0)$.
Moreover, by \eqref{eq:383}, we have
$\hat \tau''(0)=(x''(0),0,0)$.
So  
\begin{equation}\label{eq:A}
\beta_p(\mb v)=|x''(0)|
\end{equation}
is obtained. By \eqref{eq:405}, we also have 
\begin{align*}
&\det(\hat \tau'(0),\hat \tau''(0),\hat \tau'''(0)) 
=-x''(0) \Big(3x''(0)\phi_{xy}(o)+\phi_{yyy}(o)\Big),
\end{align*}
which implies 
\begin{equation}\label{eq:B}
\delta_p(\mb v)=
-\frac{3x''(0)\phi_{xy}(o)+\phi_{yyy}(o)}{x''(0)}.
\end{equation}
Since $\tau$ is an asymptotic curve, we have
\begin{align}
0=\hat\tau''\cdot N(x(y),y)
 =x'(2 \phi_{xy}+x' \phi_{xx})+\phi_{yy},
\end{align}
where the dot \lq\lq $\cdot$" means the
canonical inner product on $\R^3$.
After differentiating this equality, 
we use \eqref{eq:383} and $x'(0)=0$. 
Then we have
$$
0=2 \phi_{xy}(o)x''(0)+\phi_{yyy}(o),
$$
which implies
$x''(0)=-{\phi_{yyy}(o)}/{2 \phi_{xy}(o)}$.
By this with \eqref{eq:A} and \eqref{eq:B},
we obtain
\begin{equation}\label{eq:C2}
\beta_p(\mb v)=\left|\frac{\phi_{yyy}(o)}{2\phi_{xy}(o)}\right|,\qquad
\delta_p(\mb v)=-\phi_{xy}(o).
\end{equation}
By Propositions \ref{prop:386} and \ref{prop:R},
\eqref{eq:C2} is rewritten as
\begin{equation}\label{eq:C}
2\beta_p(\mb v)|\delta_p(\mb v)|=|\rho_{p}(\mb v)|,\qquad
|\delta_p(\mb v)|=\sqrt{-K_p}.
\end{equation}
In particular, Fact \ref{prop:B} and
the second assertion of
Proposition \ref{prop:A} and the equality
\begin{equation}\label{eq:618}
K_p=-\frac{\rho_p(\mb v)^2}{4\beta_p(\mb v)^2}
\end{equation}
are obtained.

\section{Proof of Proposition~\ref{prop:A}}
We next prove the remaining statement of Proposition~\ref{prop:A}:
The pair of asymptotic tangential curves is obtained 
as the section of the image of  $f$ by the $xy$-plane, 
that is, the union of the image of 
the two asymptotic tangential curves
coincide with the following set: 
$$
\tilde \Sigma:=\{(x,y)\in \R^2\,;\, \phi(x,y)=0\}.
$$
Let $\sigma(t)$ ($|t|<\epsilon$) be the 
generator of asymptotic tangential curve
of $f$ at $o$ in the $xy$-plane 
such that $\sigma(0)=o$
and $\sigma'(0)(\ne \mb 0)$ 
points in the $y$-axis.
So we may  write
$
\sigma(t)=(\xi(t),t)
$
where $\xi(t)$ is a smooth function 
defined around $t=0$.
By definition, we have $\xi'(0)=0$.
By the definition of $\sigma$, we have
$
\phi(\xi(t),t)=0,
$
and
\begin{align*}
& \phi_y+\phi_x\xi'=0, \,\,
 \phi_{yy}+\phi_x\xi''+2\xi'\phi_{xy}+
\phi_{xx}\xi'^2=0, \\
&
\phi_{yyy}+\phi_x\xi'''
+3\phi_{xy}\xi''+3\phi_{xx}\xi'\xi''+3\phi_{xyy}\xi'\\
&\phantom{aaaaaaaaaaaaaaaaaaaaa}+3\phi_{xxy}\xi'^2+\phi_{xxx}\xi'^3=0.
\end{align*}
The fact $\xi'(0)=0$ with
\eqref{eq:348} and \eqref{eq:383},
the third equation implies that
$3\xi''(0)=-{\phi_{yyy}(o)}/{\phi_{xy}(o)}$.
So we have
$$
\sigma'(0)=(0,1),\qquad
\sigma''(0)=\Big(-\frac{\phi_{yyy}(o)}{3\phi_{xy}(o)},0\Big),
$$
which imply
$$
\alpha_p(\mb v)=
\left. \frac{\op{det}(\sigma',\sigma'')}{|\sigma'|^3}\right|_{t=0} 
=\frac{\phi_{yyy}(o)}{3\phi_{xy}(o)}.
$$
By this with 
\eqref{eq:C2},
we have
\begin{equation}\label{eq:665b}
|\alpha_p(\mb v)|=
\frac23\beta_p(\mb v).
\end{equation}
This proves the first assertion of Proposition~\ref{prop:A}.
Moreover, by \eqref{eq:665b}
and \eqref{eq:618}, we obtain the third assertion of 
Proposition \ref{prop:A}.
Since we have proved the second assertion (at the end of Section~3), 
Proposition \ref{prop:A} is completely proved.

\section{Cuspidal curvatures of contour lines}
In this section,
we consider an embedding $f:U\to \R^3$
and set $g:=\pi_{V}\circ f$.
We fix $p\in U$.
Without loss of generality, we may assume
\eqref{eq:348}, 
\eqref{eq:383} and $K_p\ne 0$. 
So we may set $p:=(0,0)(=o)$.
Then we have
\begin{equation}\label{eq:g}
g(x,y)=(x,\phi(x,y)).
\end{equation}
Using this expression,
we compute the derivatives of the contour line (i.e. the
singular set of $g$):
We denote by
$J$ the Jacobian of $g$.
Since 
\begin{equation}\label{eq:Jg}
J=\phi_y,
\end{equation}
the fact
$\phi_{xy}(o)\ne 0$ implies that
there exists a smooth
 function $\psi(y)$ defined
around $y=0$ satisfying
$\psi(0)=0$
such that
$
\gamma(y):=(\psi(y),y)
$
gives the parametrization of the singular set 
of $g$
in the $xy$-plane.
Then 
$$
\Gamma(y):=g\circ \gamma(y)=(\psi(y),\phi(\psi(y),y))
$$
is the contour line of the image of $g$.
Since $\phi_y$ vanishes identically, we have
\begin{align*}
\Gamma'&=(\psi',\phi_x \psi'), \qquad
\Gamma''=\Big(\psi'',
\phi_x\psi''+\phi_{xy}\psi'+\phi_{xx}\psi'^2\Big), \\
\Gamma'''&=
\Big(
\psi''',
\phi_x\psi''' +2 \phi_{xy}\psi''+3 \phi_{xx}\psi'\psi'' \\
&\phantom{aaaaaaaasaaaaaaa}+\phi_{xxx}(\psi')^3+2 \phi_{xxy}\psi'^2
+\phi_{xyy}\psi'\Big).
\end{align*}
By \eqref{eq:348} with
$\psi'(0)=0$, we have
\begin{align*}
&\Gamma'(0)=(0,0),\quad
\Gamma''(0)=(\psi''(0),0), \\
&\Gamma'''(0)=(\psi'''(0),2\phi_{xy}(o)\psi''(0)).
\end{align*}

\begin{proof}[Proof of Theorem \ref{thm:292}]
Differentiating the identity
$
\phi_y(\psi(y),y)=0
$ 
twice, we have
\begin{align*}
&\phi_{xy}\psi'+\phi_{yy}=0, \\
&\phi_{xxy}\psi'^2+\phi_{xy}\psi''+
2\phi_{xyy}\psi'+
\phi_{yyy}=0.
\end{align*}
Since $\phi_{xy}(o)\ne 0$,
substitute $y=0$, we obtain
\begin{align*}
&\psi'(0)=0, \qquad
\psi''(0)=-\frac{\phi_{yyy}(o)}{\phi_{xy}(o)}.
\end{align*}
So, we have
\begin{align}\label{eq:G2G3}
\Gamma''(0)=\left(-\frac{\phi_{yyy}(o)}{\phi_{xy}(o)},0\right), \qquad
\Gamma'''(0)=\left(\psi'''(0),-2\phi_{yyy}(o)\right)
\end{align}
and 
$$
\det(\Gamma''(0),\Gamma'''(0))
=\frac{2{\phi_{yyy}(o)^2}}{\phi_{xy}(o)}.
$$
Since $\phi_{xy}(o)$ is non-zero,
$y=0$ is a cusp singular point of $\Gamma$ if and only if
$\phi_{yyy}(o)\ne 0$ (that is, $\omega_p(\mb v)\ne0$),
proving the equivalency of (a) and (c). 
On the other hand, \eqref{eq:C} implies 
the equivalency of (a) and (b).
The Jacobian of $g$ is given as \eqref{eq:Jg} under the setting
\eqref{eq:g}.
Since $K_p<0$,
\eqref{eq:532} implies $J_x(p)\ne 0$.
In particular, $p$ is a non-degenerate 
singular point (cf. Definition \ref{def:W})
of $g$.
Then, the condition that $p$ is a Whitney cusp for $g$
is equivalent to the condition 
\lq\lq$J_y(o)=0$ and $J_{yy}(o)\ne 0$" (cf. 
Fact~\ref{fact:W} in  the appendix),
which  can be rewritten as
\lq\lq$\phi_{yy}(o)=0$ and $\phi_{yyy}(0)\ne 0$".
By \eqref{eq:383}, $\phi_{yy}(o)=0$ holds.
So (d) holds if and only if $\psi_{yyy}(o)\ne 0$,
proving
the equivalency of (a) and (d).

We prove
\eqref{eq:307a} and \eqref{eq:307b}:
By
\eqref{eq:G2G3} and
\eqref{eq:305}, we have
\begin{align*}
|\omega_p(\mb v)|&(=|\mu_\Gamma|)=
2 \frac{\phi_{yyy}(o)^2}{|\phi_{xy}(o)|}
\frac{|\phi_{xy}(o)|^{5/2}}{|\phi_{yyy}(o)|^{5/2}} 
=2\frac{|\phi_{xy}(o)|^{3/2}}{|\phi_{yyy}(o)|^{1/2}}
\end{align*}
and
\begin{align*}
\rho_p(\mb v)^2\omega_p(\mb v)^4&=16\phi_{yyy}(o)^2\frac{|\phi_{xy}(o)|^{6}}{|\phi_{yyy}(o)|^2} 
=16\phi_{xy}(o)^{6}=-16K_p^3,
\end{align*}
proving \eqref{eq:307a}.
By \eqref{eq:C2} and \eqref{eq:665b},
we have
\begin{align*}
|\alpha_p(\mb v)|\omega_p(\mb v)^2
&=
\frac23\frac{|\phi_{yyy}(o)|}{2|\phi_{xy}(o)|} 
\frac{4|\phi_{xy}(o)|^{3}}{|\phi_{yyy}(o)|} 
=\frac43|\phi_{xy}(o)|^{2}=-\frac43 K_p,
\end{align*}
proving \eqref{eq:307b}.
\end{proof}

Our  map $g:=\pi_{V}\circ f$ 
($V:=df_p(\mb v)$)
admits a normal vector field $\nu:=(0,1,0)$, 
and $p\in U$ is a non-degenerate singular point of
the second kind as in  \cite{MSUY}.
Imitating the argument in  \cite[Page 272-273]{MSUY} for swallowtails,
one can show that
$|\omega_p(\mb v)|$ coincides with
the limiting singular curvature of the map $g$.

\begin{acknowledgements}
The authors thank 
Hans Havlicek, Kentaro Saji and Shintaro Akamine 
for valuable comments.
\end{acknowledgements}

\medskip
\section*{Appendix: Whitney cusps}

Consider a $C^\infty$-map
$g:(U;u,v)\to (\R^2;x,y)$. 
By using the Jacobian $J$ of $g$.
Then the singular set of $g$ is expressed as
the implicit function $J(u,v)=0$.

\begin{Definition}\label{def:W}
A singular point  $p\in U$ is called {\it non-degenerate}.
if  $(J_u,J_v)$ does not vanish at $p$.
Moreover, $p\in U$ is called a {\it Whitney cusp singular point})
if $g$ is right-left equivalent to the map germ 
$(u,v)\mapsto (u,v^3-3uv$) around $p$.
\end{Definition}

An example of Whitney cusp is given in 
the rightmost picture in Fig.~\ref{fig:2}.
The following fact is well-known (cf. \cite{W}):

\begin{Fact}\label{fact:W}
Let $p\in U$ be a point
satisfying $g_v(p)=\mb 0$.
Then $p$ is a 
 fold  $($resp. a Whitney cusp$)$
singular point if $p$ is non-degenerate
and $J_v(p)\ne 0$
$($resp. $J_v(p)=0$ and $J_{vv}(p)\ne 0)$.
\end{Fact}

\end{document}